\documentclass[12pt]{article}

\usepackage{amssymb,amsmath}
\usepackage{graphicx}
\usepackage{mathdots}
\usepackage{amsbsy}
\usepackage{amscd}
\usepackage{amsthm}
\usepackage{mathrsfs}
\usepackage{verbatim}
\usepackage[colorlinks]{hyperref}
\usepackage{authblk}
\usepackage{mathabx}
\usepackage[english]{babel}
\usepackage[T1]{fontenc}
\usepackage[utf8]{inputenc}
\usepackage{tikz}
\usepackage{fullpage}

\usepackage{textcomp}

\newtheorem{Theorem}[equation]{Theorem}
\newtheorem{Corollary}[equation]{Corollary}
\newtheorem{Lemma}[equation]{Lemma}
\newtheorem{Proposition}[equation]{Proposition}

\theoremstyle{definition}
\newtheorem{Definition}[equation]{Definition}
\newtheorem{Example}[equation]{Example}

\theoremstyle{remark}

\newtheorem{Remark}[equation]{Remark}

\numberwithin{equation}{section}

\numberwithin{figure}{section}

\newcommand{\PP}{{\mathbb P}}

\newcommand{\Z}{{\mathbb Z}}
\newcommand{\Q}{{\mathbb Q}}

\newcommand{\N}{{\mathbb N}}

\newcommand{\mf}[1]{\mathfrak{#1}}

\newcommand{\mc}[1]{\mathcal{#1}}
\newcommand{\mb}[1]{\mathbb{#1}}

\begin{document}

\title{Divisors and specializations of Lucas polynomials}
\author[1]{Tewodros Amdeberhan}
\author[2]{Mahir Bilen Can}
\author[3]{Melanie Jensen}
\affil[1]{{\small Tulane University, New Orleans; tamdeber@tulane.edu}}
\affil[2]{{\small Tulane University, New Orleans; mcan@tulane.edu}}
\affil[3]{{\small Tulane University, New Orleans; mjensen1@tulane.edu}}
\normalsize


\maketitle

\begin{abstract}
Three-term recurrences have infused stupendous amount of research 
in a broad spectrum of the sciences, such as orthogonal polynomials (in special functions)
and lattice paths (in enumerative combinatorics).
Among these are the Lucas polynomials,
which have seen a recent true revival. In this paper one of the themes of investigation 
is the specialization to the Pell and Delannoy numbers. 
The underpinning motivation comprises primarily of divisibility and symmetry. 
One of the most remarkable findings is a structural decomposition 
of the Lucas polynomials into what we term as flat and sharp analogs.
\end{abstract}

\section{Introduction}

In this paper, we focus on two themes on Lucas polynomials, the first of which has a rather ancient flavor. 
In mathematics, often, the simplest ideas carry most importance, and hence they live longest.
Among all combinatorial sequences, the (misattributed) Pell sequence seem to be particularly 
resilient. Defined by the simple recurrence 
\begin{align}\label{Pell recurrence}
P_n = 2P_{n-1} + P_{n-2}\ \text{ for } n\geq 2,
\end{align}  
with respect to initial conditions $P_0=0$, $P_1=1$, Pell numbers appear in 
ancient texts (for example, in Shulba Sutra, approximately 800 BC).
The first eight values of $P_n$ are given by $(0, 1, 2, 5, 12, 29, 70, 169)$, and the remainders modulo 3 are 
\begin{align}\label{A:periods}
(P_0,P_1,P_2,P_3,P_4,P_5,P_6,P_7) \equiv_3 (0,1,2,2,0,2,1,1).
\end{align}
It is hard not to appreciate (\ref{A:periods}), since the sequence $( P_n \mod 3 : n\geq 0 )$ 
is periodic, of period 8. 
This fact is easily proven by inducting on $n$, and by using the recurrence from (\ref{Pell recurrence}).

Let $m,n$ be two positive integers, and let $p$ be a prime number.
By the Fundamental Theorem of Arithmetic, there exists a unique expression of the form 
$m/n= p_1^{f_1}\cdots p_r^{f_r}$ for some integers $f_i \in \Z$, and prime numbers $p_i$.
The $p$-adic valuation of $m/n$ is then defined by
$$
\nu_p \left(\frac{m}{n} \right) = 
\begin{cases}
f_i & \text{ if } p= p_i, \\
0 & \text{ otherwise.}
\end{cases}
$$
Although the ancients did not document their $p$-adic arithmetic, 
it is fair to assume that the tools for proving the following interesting consequence of the 8-periodicity 
was at their disposal: the 3-adic valuation of the Pell sequence is of the form 
\begin{align}\label{A:3-adic}
\nu_3 (P_n ) =
\begin{cases}
\nu_3(3k)& \text{ if } n=4k, \\
0  & \text{ otherwise. }
\end{cases}
\end{align}
Indeed, the case $4\nmid n$ is evident from the periodicity and (\ref{A:periods}). 
For the other cases, we use the following well-known consequence
\begin{equation}\label{E:Pell relation}
P_{m+n}=P_mP_{n+1}+P_{m-1}P_n
\end{equation}
of the recurrence (\ref{Pell recurrence}).

Suppose $n=4(3k+1)$. From (\ref{E:Pell relation}), $P_{12k+4}=P_{12k}P_5+P_{12k-1}P_4$. 
By induction, $\nu_3(P_{4(3k)})=2+\nu_3(k), \nu_3(P_{12k-1})=0$. 
By direct calculation $\nu_3(P_5)=0, \nu_3(P_4)=1$. So, 
$$
\nu_3(P_n)=1=\nu_3(4(3k+1)).
$$
Suppose $n=4(3k+2)$. From (\ref{E:Pell relation}), $P_{12k+8}=P_{12k}P_9+P_{12k-1}P_8$. 
By induction, $\nu_3(P_{4(3k)})=2+\nu_3(k), \nu_3(P_{12k-1})=0$. 
By direct calculation $\nu_3(P_9)=0, \nu_3(P_8)=1$. So, 
$$
\nu_3(P_n)=1=\nu_3(4(3k+2)).
$$

\noindent
Suppose $n=4(3k+3)$. Once more, apply (\ref{E:Pell relation}) repeatedly to obtain
$$
\begin{aligned} 
P_{12k+12}
&=P_{8k+8}P_{4k+5}+P_{8k+7}P_{4k+4} \\
&=(P_{4k+4}P_{4k+5}+P_{4k+3}P_{4k+4})P_{4k+5}+P_{8k+7}P_{4k+4} \\
&=P_{4k+4}[(P_{4k+5}+P_{4k+3})P_{4k+5}+P_{8k+7}] \\
&=P_{4k+4}[(P_{4k+5}+P_{4k+3})P_{4k+5}+P_{4k+3}P_{4k+5}+P_{4k+2}P_{4k+4}] \\
&=P_{4k+4}[2(P_{4k+4}+P_{4k+3})P_{4k+5}+P_{4k+3}P_{4k+5}+P_{4k+2}P_{4k+4}] \\
&=P_{4k+4}[(2P_{4k+5}+P_{4k+2})P_{4k+4}+3P_{4k+3}P_{4k+5}]. 
\end{aligned}
$$
Since $\nu_3(3P_{4k+3}P_{4k+5})=1, \nu_3(P_{4k+4})=1+\nu_3(k+1)$ and $3\mid (2P_{4k+5}+P_{4k+2})$, 
it follows that the terms in $[(2P_{4k+5}+P_{4k+2})P_{4k+4}+3P_{4k+3}P_{4k+5}]$ are divisible by {\em exactly} $3$. 
Combining these facts, 
$$
\nu_3(P_n)=\nu_3(P_{4k+4})+1=\nu_3(3(k+1))+1=\nu_3(4(3k+3)).
$$ 
The proof of (\ref{A:3-adic}) is complete.

We denote by $\N$ the set of all non-negative integers, and by $\PP$ 
the set of positive integers. 
\begin{Corollary}\label{C:Motivating}
Given $k\in \PP$, let $n=4k$. Then $P_n^2$ does not divide $P_{n^2}$.
\end{Corollary}
\begin{proof}
A simple use of~(\ref{A:3-adic}) leads to $\nu_3(P_{n^2}) = 1+ 2 \nu_3 (k) 
< \nu (P_n^2)= 2+ 2\nu_3(k)$.
\end{proof}
The hypothesis of Corollary~\ref{C:Motivating} is restrictive in the sense that $n$ is assumed to be 
a multiple of 4. Our effort to remove the restriction has led us to consider the same question in 
a more general context, for a family of polynomials $L_n=L_n(s,t)\in \N[s,t]$, known as 
{\em Lucas polynomials},\footnote{In \cite{SS1}, $L_n$ is denoted by $\{ n \}$.}
defined by 
$$
L_n = s L_{n-1} + t L_{n-2},\ \text{ subject to the initial conditions }\ L_1=1,L_0=0.
$$
Obviously, when $s=2$, $t=1$ we recover Pell numbers. At the same time, 
Lucas polynomials have many other interesting specializations:
\begin{enumerate}
\item $L_n (1,1)= f_n$, $n$-th Fibonacci number;
\item $L_n(2,-1) = n$, for all $n\geq 0$; 
\item $L_n(s,0) = s^{n-1}$, for all $n\geq 1$;
\item $L_{2n}(0,t) =0$, and $L_{2n+1}(0,t)=t^n$, for all $n\geq 0$;
\item $L_n(q+1,-q) = 1+ q+\cdots +q^{n-1}$, the standard $q$-analog of $n$. 
\end{enumerate}
The main result that motivated our paper is the following truly remarkable multiplicity-free property
of Lucas polynomials: 
\begin{Theorem}
Let $d\neq 1$ be a divisor of $n\in \PP$. Then $L_d^2$ does not divide $L_n$.
\end{Theorem}
\noindent
Note that, by evaluating $L_n$ at $s=2,t=1$, we obtain Corollary~\ref{C:Motivating} without any 
restriction on $n$.

From an algebraic point of view, ``binomial coefficients'' are the special values of the 
$\Q$-valued function  
\begin{align}\label{A:in rational}
{x_n \choose x_k} = 
\frac{x_n x_{n-1}\cdots x_{n-k+1}}{x_{k} x_{k-1}\cdots x_2 x_1},\ (\text{when } 1 \leq k \leq n) 
\end{align}
defined on a sequence $(x_i)_{i\in \PP}$ of non-negative integers $x_i$.
For an arbitrary integer sequence, the binomials in~(\ref{A:in rational})
need not be integral. However, it follows from well-known combinatorial reasons
that for the sequence $x_i = i$, for all $i\in \N$, the binomial coefficients are integers. 
When $x_n$ is the $n$-th Fibonacci number, the associated binomial-like 
coefficients, customarily called {\em fibonomials}, are integers as well.

In general, to understand the nature of integer sequences, it is often helpful to study them 
by introducing extra parameters. For Fibonacci numbers there are many polynomial generalizations, 
and the family of Lucas polynomials is one of them. 
In analogy, the Lucas polynomial analog of the fibonomials 
are defined by 
$$
{L_n \choose L_k} := \frac{ L_n L_{n-1} \cdots L_{n-k+1}}{ L_k L_{k-1} \cdots L_1}.
$$
The tapestry
\begin{align}\label{SS1 recurrence} 
{ L_{m + n} \choose L_m } = L_{ n +1} { L_{m +n -1} \choose L_{m-1} } + t L_{ m-1} { L_{m+n-1} \choose L_{n-1} },
\end{align}
which is a consequence of the definitions, shows that ${ L_{m+n} \choose L_m }$ 
are indeed polynomials in $\N[s,t]$. 
Sagan and Savage in~\cite{SS1} call these expressions  
{\em lucanomial coefficients},\footnote{In \cite{SS1}, ${L_n \choose L_k}$ is denoted by ${ n \brace k}$.} 
and they furnish a combinatorial interpretation for them.

One of our goals in this paper is to better understand
these polynomials by analyzing their factorizations. To this end, suppose $n= p_1^{e_1} \cdots p_r^{e_r}$ 
is the prime factorization of $n$. 
We define the {\em $n$-th flat Lucas polynomial} to be the product 
\begin{align}\label{A:flat Lucas}
L_n^\flat = L_{p_1} L_{p_2} \cdots L_{p_r},
\end{align}
and the {\em $n$-th sharp Lucas polynomial} to be 
\begin{align}\label{A:sharp Lucas}
L_n^\sharp = \frac{L_n}{L_n^\flat}.
\end{align}
Obviously, a flat Lucas polynomial is a polynomial. Less obvious is to show that 
a sharp Lucas polynomial is indeed a polynomial (in $s$ and $t$). 
We prove this fact in Corollary~\ref{C:divisible}.

We define {\em flat} and {\em sharp} factorials in a conventional manner, as follows: 
$$
L_n^\flat ! = L_n^\flat L_{n-1}^\flat \cdots L_1^\flat\  \text{ and }\ L_n^\sharp ! 
= L_n^\sharp L_{n-1}^\sharp \cdots L_1^\sharp.
$$
Accordingly, let us introduce 
$$
{ L_n \choose L_k }^\flat = \frac{ L_n^\flat !} {L_{n-k}^\flat ! \, L_k^\flat ! }
\ \text{ and } \ { L_n \choose L_k }^\sharp = \frac{ L_n^\sharp !} {L_{n-k}^\sharp ! \, L_k^\sharp ! },
$$
and call ${ L_n \choose L_k }^\flat $ and ${ L_n \choose L_k }^\sharp$, respectively,  
{\em flat} and {\em sharp} lucanomial coefficients.
For all $0 \leq k \leq n $, we observe the following ``flat and sharp'' decomposition of lucanomials:
$$
{ L_n \choose L_ k } = { L_n \choose L_k }^\flat { L_n \choose L_k }^\sharp.
$$
What is really intriguing is that 
\begin{Theorem}
Both the flat and sharp lucanomials are polynomials in $\N[s,t]$.
\end{Theorem}
While the proof of polynomiality of ${ L_n \choose L_k }^\flat $ follows from a much more general 
fact about polynomials, when specialized to integral values of $s$ and $t$, it provided us with the following 
challenge.

Let $s$ and $t$ be two fixed integers. In this case, we denote the numerical sequence 
$(L_n^\flat(s,t))_{n\in \PP}$ by $(\text{ev}(L_n^\flat))_{n\in \PP}$ in order to distinguish 
from the polynomials $L_n^\flat$. 
Empirical evidence suggests, for a prime number $p$, that there exists a constant  
$\theta=\theta_{s,t}(p) \in \N$ such that 
$$
\nu_p( \text{ev}(L_n^\flat) !)=\left\lfloor \frac{n}{\theta} \right\rfloor.
$$
We do not pursue this question here, however the interested reader might do so.
Note that when $s=2$, $t=-1$, the number $\text{ev}(L_n^\flat)$ is nothing but $n^\flat$, 
the product of all prime numbers dividing $n$.
In this case, $\theta_{2,-1}(p)=p$, and hence 
$\nu_p( n^\flat ! ) = \left\lfloor \frac{n}{p} \right\rfloor$.

\vspace{.5cm}
\noindent
\textbf{Question:} Does there exist an explicit expression for $\theta_{s,t}(p)$ ?
\vspace{.5cm}

The second theme of our paper is on certain symmetry, 
which is lacking from Lucas polynomials. 
The specialization of $L_n$ at $s=x+1$, $t=x$ (denoted here by $D_n$) has a 
happy ending in the sense that 
\begin{Theorem}
For all $0\leq k \leq n$, the delannomial coefficient 
$$
{ D_n \choose D_k } = \frac{ D_n D_{n-1}\cdots D_{n-k+1}}{ D_k \cdots D_1}.
$$ 
is symmetric and unimodal in the variable $x$. 
\end{Theorem}
\begin{Remark}
When $x=1$, the numbers $D_n$ evaluate to Pell numbers, which were our original motivation 
for the present work.
\end{Remark}
Wishing for more, we apply divided-difference calculus to Lucas polynomials 
and obtain various interesting corollaries, one of which we mention here. 
Let $\partial_{s,t}: \N [s,t] \rightarrow \N[s,t]$ denote the operator 
$\partial_{s,t} (F(s,t))= (F(s,t)-F(t,s))/(s-t)$.
Let $\alpha\in\mathbb{N}$ and define {\em modified Lucas polynomials} 
by 
$L_0(s,t:\alpha)=L_1(s,t:\alpha)=\alpha$. For $n\geq2$, define
$$
L_n(s,t:\alpha)=sL_{n-1}(s,t:\alpha)+tL_{n-2}(s,t:\alpha).
$$
Let $S_n(s,t:\alpha)$  denote the divided-difference polynomial $\partial_{s,t}L_n(s,t:\alpha)$.
\begin{Theorem} The following hold true:
\begin{enumerate}
\item[(i)] $S_n(s,t:\alpha)=\alpha S_n(s,t:1)$ for all $\alpha\in\mathbb{N}$;

\item[(ii)] $(s+t-1)$ divides $S_n(s,t:\alpha)$ for all $n\in\mathbb{N}$;

\item[(iii)] $\frac{S_n(s,t:\alpha)}{s+t-1}$ has non-negative integral coefficients, only.
\end{enumerate}
\end{Theorem}

An important connection between multiplicative arithmetic functions and
symmetric polynomials, which we were not aware of at the time of writing this paper 
is pointed out to us by an anonymous referee. 
In the articles \cite{MacHenry00,MacHenry05,MacHenry12,MacHenry13}, MacHenry and et al 
develop the idea that the convolution algebra of multiplicative arithmetic functions is representable by the evaluations of 
certain Schur polynomials. It would be interesting to investigate our flat and sharp Lucanomials in the
context of arithmetic functions in relation with symmetric functions.

We conclude our introduction with an observation on further 
potential interpretation of the Lucas polynomials in the context of representation theory. 
We plan to pursue this in the future, so we keep it brief in here.

\noindent

Let $q$ be a variable and $\mb{K}$ denote a field of characteristic zero. 
Consider the polynomial ring $\mc{P}=\mb{K}[q][x_1,\dots,x_n]$ in $n$ variables 
over the ring $\mb{K}[q]$. 
If  $\sigma_i : \mc{P} \rightarrow \mc{P}$, $1\leq i<n$, denotes the $\mb{K}[q]$-linear operator  
interchanging $x_i$ with $x_{i+1}$, define the operators on the ring $\mc{P}$ by
$$
T_i=(q-1)\left[\frac{x_i-x_{i+1}\sigma_i}{x_i-x_{i+1}}\right]+\sigma_i \qquad (1\leq i<n).
$$
Then the $T_i$'s generate a faithful representation of a particular deformation $\mc{H}_n$ of 
the group ring $\mb{K}[\mf{S}_n]$ of the symmetric group $\mf{S}_n$. In fact, it is isomorphic 
to a specialization of the Iwahori-Hecke algebra of $\mf{S}_n$.

Let $\rho_{(n-1,1)}$ denote the irreducible  representation of $\mc{H}_n$ on 
the space $V$ of linear polynomials without constant terms modulo $x_1+\cdots +x_n = 0$,  
having the polynomials $\{x_{n-1}+\cdots+x_1,\dots,x_2+x_1,x_1\}$ as a basis.
Consider the following element of $\mc{H}_n$:
$$
H=\sum_{i=1}^{n-1}(T_i-q).
$$

If $\rho_{(n-1,1)}(H)$ is the image of $H$ under the representation $\rho_{(n-1,1)}$ 
with respect to the above basis, then the matrix form of the image is 
$\rho_{(n-1,1)}(H)=M_{n-1}(q)-(1+q)I_{n-1}$, where $I_{n-1}$ is the identity matrix, 
and $M_n(q)$ is the tri-diagonal matrix (with super-diagonal all $q$'s, diagonal all $0$'s, 
sub-diagonal all $1$'s, and everything else $0$). 
For example, 
\[
   \rho_{(4,1)}(H)=\left[ 
   {\begin{array}{cccc}
   -(1+q) & q & 0 & 0  \\ 1 & -(1+q) & q & 0 \\ 0 & 1 & -(1+q) & q \\ 0 & 0 & 1 & -(1+q)  
  \end{array} } 
  \right].
\]

Furthermore, the characteristic polynomial of $\rho_{(n-1,1)}(H)$ takes the form  
$Ch_{n-1}(x,q)=\det[(x+1+q)I_{n-1}-M_{n-1}(q)]$. If we replace $q=t$ and $s=x+1+q$, 
then $Ch_{n-1}(s,t)=\det[sI_{n-1}-M_{n-1}(t)]$. These determinants are easy to compute recursively by
$$
Ch_n=sCh_{n-1}+tCh_{n-2}.
$$
Comparing initial conditions reveals a surprising connection: $Ch_{n-1}(s,t)=L_n(s,t)$, the Lucas polynomials!

\newpage

\section{Preliminaries}

A closely related family of polynomials, defined by the same recurrence
$K_n = s K_{n-1} + t K_{n-2}$ with respect to the initial conditions $K_0=2,K_1=s$
is called the family of {\em circular Lucas polynomials}.\footnote{In \cite{SS1}, $K_n$ is denoted by $\langle n \rangle$.}
The ordinary and circular Lucas polynomials are interwoven by the identity: 
\begin{equation} \label{mn}
2L_{m + n} = K_n L_m + K_mL_n \ \text{ for all } m,n\in \N.
\end{equation}
Table~\ref{List} gives a short list of $K_n$'s and $L_n$'s for small $n$.
\begin{table}[htdp]
\begin{center}
\begin{tabular}{l|l}
Lucas polynomials& Circular Lucas Polynomials\\
\hline
$L_0 =0$ & $K_0 =2$ \\
$L_1 =1$ & $K_1 =s$ \\
$L_2 =s$ & $K_2 =s^2+2t$ \\
$L_3 =s^2+t$ & $K_3 =s^3+3st$ \\
$L_4 =s^3+2st$ & $K_4 =s^4+4s^2t+2t^2$ \\
$L_5 =s^4+3s^2t+t^2$ & $K_5 =s^5+5s^3t+5st^2$ \\
$L_6 =s^5+4s^3t+3st^2$ & $K_6 =s^6+6s^4t+9s^2t^2+2t^3$ \\
\end{tabular}
\end{center}
\caption{A list of Lucas and circular Lucas polynomials}
\label{List}
\end{table}%
Due to their recursive nature, the polynomials $K_n$ and $L_n$, as well as 
$L_n \choose L_k$ have nice combinatorial interpretations:
\begin{enumerate}
\item For all $n\geq 1$,
\begin{equation*}
K_n = \sum_{T \in \mathcal{C}_{n}}w(T),
\end{equation*}
where $w(T) = s^{m}t^{d}$ such that $m$ is the number of monominos 
and $d$ is the number of dominos and 
$\mathcal{C}_{n}$ is the set of all circular tilings of a 
$1\times n$ rectangle with disjoint dominos and monominos. 
 
\item For all $n\geq 1$, 
\begin{equation}
L_n= \sum_{T \in \mathcal{L}_{n-1}}w(T)
\end{equation}
where $w(T) = s^{m}t^{d}$ such that $m$ is the number of monominos 
and $d$ is the number of dominos and 
$\mathcal{L}_{n-1}$ is the set of all linear tilings of a $1\times (n-1)$ 
rectangle with disjoint dominos and monominos. 
\item 
For a partition $\lambda$, let $\mc{L}_\lambda$ denote the set of all possible linear 
tilings of the rows of the Young diagram of $\lambda$, and for $\lambda\subseteq m\times n$,
let $\lambda^*$ denote the the complimentary Young diagram of $\lambda$ in $m \times n$. 
Also, the let $\mc{L}'_{\lambda}$ denote the set of all linear tilings of the rows of $\lambda$ 
that do not start with a monomino. 
Finally, the weight $w(T)$ of an element $T=(T_1,T_2) \in \mc{L}_\lambda \times \mc{L}_\mu$ 
is defined as the product of the weights of the rows of $T_1$ and $T_2$. 
It is shown in [Theorem 3,~\cite{SS1}] that 
if $m$ and $n$ are two positive integers, then 
\begin{align}\label{L:SS1 Theorem 3}
{L_{m+n} \choose L_n } = \sum_{\lambda \subseteq m\times n } \sum_{T\in \mc{L}_\lambda \times \mc{L}'_{\lambda^*}} w(T). 
\end{align}
\end{enumerate}

\section{Prime Divisors of Lucas Polynomials}

\begin{Proposition}\label{neven}
Let $N$ be a positive integer. Then $N$ is even if and only if 
$L_2$ divides $L_N$. Moreover, 
\begin{equation} \label{ctol}
\frac{L_{2N}}{L_N} = K_N \text{ for any } N \geq 1.
\end{equation}
\end{Proposition} 

\begin{proof}
Equation~(\ref{ctol}) is immediate from~(\ref{mn}).
When $N$ is even, the linear tiling corresponding to $L_N$ 
has length $N-1$, which is odd. Thus, each linear tiling must 
contain at least one monomino.
The converse statement is easy to show by induction 
and the recurrence for $L_N$.

\end{proof}

\begin{Corollary}\label{C:even or odd}

Let $N = 2^r n$ for some positive integers $r,n$. Then  
\begin{equation*}
L_N = L_n \prod_{i = 1}^r K_{\frac{N}{2^i}}.
\end{equation*}
In particular, when
$ N = 2^r$ with $r \geq 2$, we have $L_N=  \prod_{i = 1}^r K_{ \frac{N}{2^i}}$.

\end{Corollary}

\begin{proof}
This follows from a repeated use of Proposition~ \ref{neven}.
\end{proof}

\vspace{.1cm}

\begin{Example}
When $n = 6$:
\begin{align*}
\frac{L_6}{L_3} & = \frac{s^{5} + 4s^{3}t + 3st^{2}} {s^2+t}\\
& = \frac{(s^3+3st)(s^2+t)}{s^2+t}\\
& = K_3
\end{align*}
\end{Example}

\begin{Theorem}\label{T:divisible}
Let $N$ be a positive integer. Then 
\begin{enumerate}
\item[(i)] If $a \mid N$, then $L_a \mid L_N$. More precisely,  
\begin{equation}\label{E:division}
\frac{L_N}{L_a} =  \sum_{i=1}^b \frac{ K_{N-ia }  K_a^{i-1}}{2^i}.
\end{equation}
\item[(ii)] If $L_a \mid L_N$, then $a \mid N$.
\end{enumerate} 
\end{Theorem}

\begin{proof}

To prove $(i)$, it suffices to prove the identity~(\ref{E:division}).
If $a \cdot b = N$, we write 
\begin{align}
L_N &= L_{a + (N-a)} \notag \\
& =   \frac{ K_{ N-a} }{2} L_a + \frac{ K_a}{2} L_{N-a}. \label{SS2}
\end{align}
Since $N-ia = N-(i+1)a + a$ for any $i=1,\dots, b$, 
we repeatedly use~(\ref{mn}) in (\ref{SS2}) to get:
\begin{align}
L_N &=  \sum_{i=1}^b \frac{ K_{N-ia} }{2^i} L_a K_a^{i-1}.
\end{align}

For part $(ii)$, we already know from Proposition~\ref{neven} that 
our claim is true when $a=2$, so we assume that $a>2$. 

Observe that $L_N$ at ${s=t=1}$ is the $N$-th Fibonacci number $f_N$. 
Thus, if $L_a$ divides $L_N$, then $a$-th Fibonacci number $f_a$ divides $f_N$. 
On the other hand, it is well known that, for $n >2 $, $f_n \mid f_N$ if and only if 
$n \mid N$ (see~\cite{BH}). Hence, the proof is complete.

\end{proof}

\begin{Example}
$$
L_6= s^{5} + 4s^{3}t + 3st^{2} = L_2 L_3 (s^2+3t).
$$
\end{Example}

\begin{Corollary}\label{C:divisible}
Let $N$ be a positive integer with prime factorization 
$N= p_1^{e_1} \cdots p_r^{e_r}$, where $e_1,\dots, e_r$ are some positive integers. Then 
$L_N$ is divisible by $\prod_{i=1}^rL_{ p_i }$.
\end{Corollary}

\begin{proof}
By Theorem~\ref{T:divisible} and induction, we re-write $L_N$
in the form $L_N = L_{p_1} \cdots L_{p_{r-1}} p(s,t)$
for some polynomial $p(s,t)$.

Now, if a prime factor of the polynomial $L_{ p_r}$ divides any 
of $L_{p_1},L_{p_2},\dots. L_{ p_{r-1}}$, then a prime factor of the $p_r$-th
Fibonacci number divides one of $f_{p_1},\dots, f_{p_{r-1}}$. 
However, it is well known that Fibonacci numbers that have a prime index 
do not share any common divisor greater than 1, since ~\cite{R}
\begin{align}\label{fibonacci common divisors}
\text{gcd} (f_n, f_m) = f_{\text{gcd}}(n,m).
\end{align}
Therefore, $L_{p_{r}}$ divides $p(s,t)$, and hence, the proof is complete. 
\end{proof}

Although $L_2$ divides $L_8$, it is not true that higher powers of $L_2$ 
divide $L_8$:
$$
\frac{ L_8 } { L_2^2 }  = \frac{ (s^2 + 2t) (s^4 + 4s^2 t + 2t^2)}{ s}.
$$
More generally, in our next result we are going to show that $L_N$ 
is not divisible by the square of any of its divisors.

\newpage

\begin{Theorem}\label{T:p^2 does not divide}
Let $p\neq 1$ be a (not necessarily prime) divisor of $N\in \PP$. 
Then $L_{p}^2$ does not divide $L_N$. 
\end{Theorem}

\begin{proof}
Let $n$ be such that $N=np$. We claim that 
\begin{align}\label{CLAIM}
L_N \equiv n t^{n-1} L_{ p-1}^{n-1} \mod L_p^2.
\end{align}
We show this by proving that $L_N/L_p \equiv n t^{n-1} L_{ p-1}^{n-1} \mod L_p$.
Obviously, if $n=1$, then there is nothing to prove. To use induction, assume that 
our claim is true for $n$. 
After some cancellations, equation (\ref{SS1 recurrence}) implies that 
\begin{align}\label{basic basic}
L_{ a+b } = L_a L_{ b+1} + t L_{a-1} L_{ b} \text{ for all } a,b\geq 0.
\end{align}
Replacing $a$ by $np$ and $b$ by $p$ in (\ref{basic basic}), we have 
$L_{ np+p } = L_{ np } L_{ p+1 } + t L_{np-1} L_{ p}$.
Combining this with the defining recurrence $L_p= s L_{ p-1} + t L_{p-2}$, 
induction assumption and one more application of (\ref{basic basic}), we get: 
\begin{align*}
\frac{ L_{ (n+1)p } }{ L_p} &\equiv \frac{ L_{ np } }{L_p} L_{ p+1 } + t L_{np-1} \mod L_p\\
&\equiv n t^{n-1} L_{ p-1}^{n-1} \left( s L_p+ t L_{p-1} \right) + t L_{np-1} \mod L_p \\
&\equiv n t^{n} L_{ p-1}^{n}  + t L_{np-1} \mod L_p.
\end{align*}
Thus, it remains to show that $L_{np-1} \equiv t^{n-1} L_{p-1}^n \mod L_p$.
We use induction on $n$ once more. If $n=1$, there is nothing to prove. 
Assuming validity for $n$ and using (\ref{basic basic}) once again, we see that  
$$
L_{ np + p -1 } = L_{np}L_{p}+tL_{np-1}L_{p-1} \equiv t^{n} L_{ p-1 }^n \mod L_p.
$$
Hence, we have our claim proven. 

Since $L_{p-1}$ is not divisible by $L_p$ as $p$ and $p-1$ are relatively prime, 
we see that the right hand side of (\ref{CLAIM}) is not zero, hence $L_N$ is not 
divisible by $L_{ p }^2$.

\end{proof}

\section{Flat and Sharp Decomposition}

\subsection{Flat and Sharp Lucas polynomials}

We know from Corollary~\ref{C:divisible} that the sharp Lucas polynomials are 
indeed polynomials. Due to prime involvement, finding a combinatorial interpretation 
for sharp polynomials is a challenging problem. Equivalently difficult is the problem of describing 
all monomials of a sharp (or of a flat) polynomial. 
Note that, if $n$ itself is a prime number, then $L_n^\sharp$ is trivial ($=1$). 
More generally, suppose $n=p_1^{e_1}\cdots p_r^{e_r}$ is the prime decomposition of $n$.
It is easy to see from the recursive definition of $L_n$ that $L_n$ is monic in $s$ 
with degree $n-1$ (for $n\geq 1$). 
Therefore, the $s$-degree of $L_n^\sharp$ is equal to 
$$
\deg_s L_n^\sharp= n-1 - \sum_{i=1}^r (p_i-1) = n- \sum_{i=1}^r p_i + r - 1.
$$
For the $t$-degree, we have 
$$
\deg_t L_n^\sharp = \left\lfloor \frac{n-1}{2} \right\rfloor  - \sum_{i=1}^r \left\lfloor \frac{p_i-1}{2} \right\rfloor.
$$

When $N\in \PP$ is a power of 2, $L_N^\sharp$ reveals itself rather explicitly. 
Indeed, we have a precise analogue of Corollary~\ref{C:even or odd}:
suppose $N = 2^r n$ for some positive integers $r,n$. Then  
\begin{equation*}
L_N^\sharp= \frac{ L_n^\sharp }{L_2} \prod_{i = 1}^r K_{ \frac{N}{2^i} }.
\end{equation*}
In the special case when
$ N = 2^r$ for $r \geq 2$, then
\begin{equation*}
L_N^\sharp =  \frac{\prod_{i = 1}^r K_{2^i} }{L_2}.
\end{equation*}

\begin{Lemma}\label{L:orthogonal}
For any prime number $p$, and an arbitrary positive integer $N$, we have
$$
\text{gcd}(L_p,L_N^\sharp)=1.
$$
\end{Lemma}
\begin{proof}
If $p$ does not divide $N$, then there is nothing to prove. So, we proceed 
with the assumption that $p$ divides $N$. Suppose $N=np$ for some $n\in \N$,
and let $g=g(s,t)\in \N[s,t]$ denote $\text{gcd}(L_p,L_N^\sharp)$. 
Obviously, we may assume that $g$ is a non-constant polynomial. 
It is also evident that $g$ is a divisor of $L_N/L_p$.
We know from the proof of Theorem~\ref{T:p^2 does not divide} that 
$L_N/L_p \equiv n t^{n-1} L_{ p-1}^{n-1} \mod L_p$, hence, 
\begin{align}\label{A:tL}
\frac{L_N}{L_p} \equiv n t^{n-1} L_{ p-1}^{n-1} \mod g.
\end{align}
Therefore, $g$ divides the right hand side of (\ref{A:tL}). 
In particular, specializing at $s=t=1$, we see that $g(1,1)$ divides $f_{p-1}^{n-1}$, hence,
a prime factor of $g(1,1)$ divides $f_{p-1}$. But this means $f_p=L_p(1,1)$ and 
$f_{p-1}$ have a common prime divisor, which is absurd. Therefore $g=1$. 
\end{proof}

Recall that $\text{gcd}(L_m,L_n) = L_{\text{gcd}(m,n)}$. 
Divisibility properties of Lucas polynomials carry over to the flattened and sharpened versions: 
\begin{Theorem}
Let $m$ and $n$ be two positive integers such that $m \mid n$. Then 
\begin{enumerate}
\item[(i)] $L_m^\flat \mid L_n^\flat  \text{ in } \N[s,t]$,
\item[(ii)] $L_m^\sharp \mid L_n^\sharp \text{ in } \N[s,t]$.
\end{enumerate}
\end{Theorem}
\begin{proof}
Part $(i)$ follows from Theorem~\ref{T:divisible}. Part $(ii)$ follows from part $(i)$ combined with 
Lemma~\ref{L:orthogonal}.
\end{proof}

\subsection{Flat and Sharp Lucanomials}

\begin{Theorem}\label{T:orthogonal}
For all $0 \leq k \leq n $, we have 
$$
{ L_n \choose L_ k } = { L_n \choose L_k }^\flat { L_n \choose L_k }^\sharp.
$$
\end{Theorem}
\begin{proof}
This is immediate from 
$$
L_n^\sharp!=\frac{L_n!}{L_n^\flat!},
$$
which itself is a consequence of equation (\ref{A:sharp Lucas}).
\end{proof}

Clearly, the remarkable combinatorial interpretation (\ref{L:SS1 Theorem 3})
of ${ L_n \choose L_k }$ exists because of polynomiality.
A natural question to ask at this point is whether or not the flat/sharp lucanomials 
are polynomials. The answer is affirmative.

We proceed with a rather general result on ``binomial coefficients'' for the flattened  
polynomial sequences. Although we state this for polynomials only, 
it stays valid for sequences in an integral domain.
\begin{Theorem}\label{T:flat is polynomial}
Let $R$ be a polynomial algebra over a field of characteristic zero. 
Let $\{ P_n \}_{n\in \N}$ be a sequence of polynomials from $R$ with $P_0=0$
and $P_1=1$. 
For each $n\in \PP$, let $P_n^\flat$ denote the flattening of $P_n$, that is $P_n^\flat=P_{p_1}\cdots P_{p_r}$, 
whenever $n= p_1^{e_1}\cdots p_r^{e_r}$ is the prime factorization of $n$.
Then the associated flat binomial 
${ P_n \choose p_k }^\flat:= \frac{P_n^\flat \cdots P_{n-k+1}^\flat }{ P_k^\flat \cdots P_1^\flat}$ 
is a polynomial. 
\end{Theorem}

\begin{proof}
If $p\in \PP$ is a prime number, then with an abuse of terminology call $P_p$ ``prime.''
We define the $P_p$-valuation of $P_n$ to be the highest exponent of $P_p$ in 
the factorization of $P_n$ in $R$. Since $P_n^\flat$ is a product of primes, 
the $P_p$-valuation of $P_n^\flat ! := P_n^\flat P_{n-1}^\flat\cdots P_1^\flat$ is equivalent to the 
$p$-adic valuation of $n^\flat !$, which is 
$$
\nu_{P_p} (P_n^\flat!)= \nu_p( n^\flat !)= \left \lfloor \frac{n}{p} \right\rfloor,
$$
hence 
\begin{align}\label{A:buyuk}
\nu_{P_p} \left( { P_n \choose P_k }^\flat \right) = \left \lfloor \frac{n}{p} \right\rfloor 
-\left \lfloor \frac{k}{p} \right\rfloor - \left \lfloor \frac{n-k}{p} \right\rfloor \geq 0.
\end{align}
To prove the inequality in (\ref{A:buyuk}) write $n= mp+r$, $k=lp+t$ where $0\leq t,s \leq p$.
So, 
$$
 \left \lfloor \frac{n}{p} \right\rfloor 
-\left \lfloor \frac{k}{p} \right\rfloor - \left \lfloor \frac{n-k}{p} \right\rfloor 
= m-l - (m-l)- \left \lfloor \frac{r-t}{p} \right\rfloor = -  \left \lfloor \frac{r-t}{p} \right\rfloor \geq 0
$$
since $r-t< p$ (possibly negative). In fact, this argument shows that 
$\nu_{P_p} \left( { P_n \choose P_k }^\flat \right) $ is 0 or 1.
Therefore, the {\it a priori }rational function ${P_n \choose P_k}^\flat$ is a polynomial. 
\end{proof}

\begin{Theorem}\label{T:sharp is polynomial}
Both the flat and sharp lucanomials are polynomials in $\N[s,t]$.
\end{Theorem}

\begin{proof}
Polynomiality of the flat lucanomials follows from Theorem~\ref{T:flat is polynomial},
so, we proceed with the sharp lucanomials. 
Since ${L_n \choose L_k}$ is a polynomial, by Theorem~\ref{T:orthogonal} it is enough
to show that the denominator of ${L_n \choose L_k}^\sharp$ has no divisor
shared with the polynomial ${L_n \choose L_k}^\flat$. In light of Lemma~\ref{L:orthogonal} 
this is now obvious.
\end{proof}

\subsection{Catalanomials}

In this section we extend the discussion to an $s,t$-version of the 
classical  Catalan numbers.

\begin{Definition} A general binomial version of Catalan, the $(s,t)$-Catalan, is defined to be
$$
C_{L_n}:=\frac{1}{L_{n+1}} {L_{2n}\choose L_n}.
$$
The {\em flat} and {\em sharp} $(s,t)$-Catalan polynomials, $C_n^\flat$, $C_n^\sharp$ 
are defined similarly, by replacing $L_i$'s with $L_i^\flat$'s, and with $L_i^\sharp$'s, respectively.
\end{Definition}

\begin{Theorem} 
The $(s,t)$-Catalans $C_{L_n},C_n^\flat,C_n^\sharp$ are all polynomials in $\mathbb{N}[s,t]$.
\end{Theorem}
\begin{proof} 
The first assertion is immediate from
$$
C_{L_n}={ L_{2n-1}\choose L_{n-1}} +t {L_{2n-1} \choose L_{n-2}},
$$
which is one of the properties found in~\cite{SS1}. 
The proof is completed by double induction on $n$ and $k$.
For the second, it is enough to observe that $L_{n+1}^\flat$ divides ${L_{2n}\choose L_n}$
and $\text{gcd}(L_{n+1}^\flat, {L_{2n}\choose L_n}^\sharp) =1$ (by Lemma~\ref{L:orthogonal}). 
The proof of the third assertion follows from that of the second. 
\end{proof}

\section{Delannomials}

Let $a,b\in \PP$ be two positive integers. 
The {\em Delannoy number} $D(a,b)$ is the number of lattice paths 
starting at $(0,0)$ and ending at $(b,a)$ moving with unit steps $(1,0),(0,1)$, or $(1,1)$. 
These numbers are given by the recurrence relation
\begin{align}\label{D:basic recurrence}
D(a,b)=D(a-1,b)+D(a,b-1)+D(a-1,b-1)
\end{align}
and the initial conditions $D(a,0)=D(0,b)=D(0,0)=1$.
The beautiful symmetry of the generating series 
$$
\mc{D}(x,y)= \sum_{
\begin{subarray}{l}
      a+b \geq 1 \\ a,\, b\, \in \N
      \end{subarray}}
D(a,b)x^{a}y^{b} = \frac{1}{1-x-y-xy}
$$ 
is indicative of a rich combinatorics associated with these numbers. 
Of particular interest is the paper \cite{Bump}, where Delannoy numbers 
find a prominent place in number theory (especially, in the discussion on the 
notion of the so-called {\it local Riemann Hypothesis}).

Let $x$ be a new variable,
and define the polynomial $D_n(x)$, $n\in \N$ by
\begin{align*}
D_n(x) = L_n\vert_{s=x+1,t=x}.
\end{align*}
It is immediate from the defining recurrence of Lucas polynomials that $D_{0}=0$, $D_{1}=1$ and 
\begin{align}\label{D:recurrence}
D_n(x)= (x+1) D_{n-1} + x D_{n-2},
\end{align}
for $n\geq 2$. 
If there is no danger of confusion we remove the argument $x$ and write simply $D_n$
in place of $D_n(x)$.

\vspace{.5cm}

The next result shows that the coefficients of $D_n(x)$ are the classical Delannoy numbers.
\begin{Lemma}\label{L:dn}
For each $n\geq 1$, we have  
\begin{align}\label{A:explicit form}
D_n(x) =  \sum_{i=1}^{n} D(n-i,i-1) x^{i-1}.
\end{align}
\end{Lemma}

\begin{proof}
Write $D_n = \sum_{i=0}^{n-1} d^n_i x^i$. Then by the recurrence~(\ref{D:recurrence})
we see that 
$$
\sum_{i=0}^{n-1} d^n_i x^i = (x+1) \sum_{i=0}^{n-2} d^{n-1}_i x^i + x \sum_{i=0}^{n-3} d^{n-2}_i x^i,
$$
or equivalently, for $1\leq i \leq n-3$, 
\begin{align}\label{A:d}
d^n_{i} = d^{n-1}_{i-1} + d^{n-1}_{i} + d^{n-2}_{i-1}.
\end{align}
Assume by induction that $d^n_i = D(n-i,i-1)$. Then the recurrence (\ref{D:recurrence}) together 
with the induction hypothesis implies that 
$$
D(n-i,i-1)= D((n-1)-(i-1),i-2)+ D((n-1)-i,i-1)+D((n-2)-(i-1),i-2),
$$
which is equivalent to (\ref{A:d}). 

\end{proof}

\begin{Lemma}
Preserve the notation from the proof of Lemma~\ref{L:dn}, and 
write $D_n = \sum_{i=0}^{n-1} d^n_i x^i$. Then $d^n_i = d^n_{n-i-1}$.
\end{Lemma}
\begin{proof}
By Lemma~\ref{L:dn}, we know that $d^n_{n-i-1} = D(n-1-(n-i-1), n-i-1)$
and that $d^n_i = D(n-1-i,i)$. Obviously these are equal quantities. 
\end{proof}

\begin{Definition}
The $(n,k)$-th {\em delannomial}
is defined to be 
$$
 { D_n \choose D_k } = \frac{ D_n D_{n-1}\cdots D_{n-k+1}}{ D_k \cdots D_2 D_1}.
$$
\end{Definition}

Let $p(x) = \sum_{i=0}^r a_i x^i$ be a polynomoial. If 
$r$ is odd, then the {\em central monomial} of $p(x)$ is defined to be 
$a_j x^j$, where $j= (r+1)/2$. If $r$ is even, it is defined to be $a_j x^j$ with 
$j= \lfloor r/2 \rfloor$.


\begin{Theorem}
For all $m$ and $n$, the delannomial ${D_{m+n} \choose  D_m}$ is symmetric and unimodal. 
\end{Theorem}
\begin{proof}
The recurrence~(\ref{SS1 recurrence}) induces the same recurrence on ${D_{m+n} \choose D_m}$.
Since product of symmetric and unimodal polynomials is symmetric and unimodal, 
we only need to show that the degree of the central monomial of 
$D_{m+1} {D_{m+n} \choose D_{m-1}}$ matches with the central monomial of 
$xD_{m-1} {D_{m+n} \choose D_{n-1}}$. This follows from induction. 

\end{proof}

\begin{Remark}
For ${D_{m+n} \choose  D_m}$, there exists a combinatorial interpretation, along the lines of~\cite{SS1}, 
by using dominos (weighted by $x$) and two kinds of monominos
(weighted by $x$ and $x^2$).
\end{Remark}

\section{Divided-Differences}

The Lucas polynomials bring in many interesting features, 
but they fail to be symmetric in the variables $s$ and $t$. For example,
$L_2=s$. 
To remedy this deficit, we consider their behavior under the 
\emph{divided-difference operator}. To be precise, we associate the sequence of polynomials defined by
$$
S_n(s,t)=\frac{L_n(s,t)-L_n(t,s)}{s-t}.
$$
Of course, $S_n(s,t)=S_n(t,s)$ for all $n\geq0$. 
Let's record some basic properties of $S_n(s,t)$. 
The next result shows a simple algebraic relation between the two sequence of polynomials via generating functions.

\begin{Lemma}\label{L:generating function for S}
Suppose $S(x;s,t)=\sum_nS_n(s,t)x^n$ and $L(x;s,t)=\sum_nL_n(s,t)x^n$. Then
$$
S(x;s,t)=(1-x)L(x;s,t)L(x;t,s).
$$
\end{Lemma}
\begin{proof} It is well-known that $L(x;s,t)=\frac{x}{1-sx-tx^2}$. Now, proceed as follows:
\begin{align*} \frac{L(x;s,t)-L(x;t,s)}{s-t}&=\frac1{s-t}\left[\frac{x}{1-sx-tx^2}-\frac{x}{1-tx-sx^2}\right] \\
&=\frac1{s-t}\left[\frac{(s-t)(1-x)x^2}{(1-sx-tx^2)(1-tx-sx^2)}\right]. \end{align*}
The proof is complete.
\end{proof}

\begin{Corollary} There is a recurrence relation linking $L_n(s,t)$ with $S_n(s,t)$:
$$
S_n(s,t)=sS_{n-1}(s,t)+tS_{n-2}(s,t)+L_{n-1}(s,t)-L_{n-2}(s,t).
$$
\end{Corollary}

\begin{proof} 
Rewrite Lemma~\ref{L:generating function for S} in the form: $(1-sx-tx^2)S(x;s,t)=(x-x^2)L(x,s,t)$. 
Taking the coefficients of $x^n$ on both sides of this equation reveals that
$$
S_n(s,t)-sS_{n-1}(s,t)-tS_{n-2}(s,t)=L_{n-1}(s,t)-L_{n-2}(s,t),
$$
which is equivalent to desired conclusion.
\end{proof}
\noindent
The generating function for \emph{second order Fibonacci numbers}, as defined in 

\begin{center}{\tt http://oeis.org.A010049},\end{center}

\noindent
is $x(1-x)/(1-x-x^2)^2$. The next statement connects these numbers with the divided-differences $S_n(1,1)$.

\begin{Corollary} Let $a_n$ denote the specialization of $S_n(s,t)$ at $s=t=1$. Then
\begin{enumerate}
\item[(i)] $a_n$ is the $(n-1)$-th second order Fibonacci number;

\item[(ii)] $a_n=f_{n-1}+\sum_{k=0}^{n-2}f_{n-2-k}f_k$.
\end{enumerate}
\end{Corollary}
\begin{proof} 
$(ii)$ Recall that $L(x;1,1)=\sum_nf_nx^n$, where $f_n$ are the Fibonacci numbers. 
Observe also that given any $F(x)=\sum_nc_nx^n$, the partial sums $\sum_{k=0}^nc_n$ 
have generating function $\frac{F(x)}{1-x}$.  
From Lemma~\ref{L:generating function for S}, we have $\frac{S(x;1,1)}{1-x}=L(x;1,1)^2$. 
Extract the coefficients of $x^n$ to obtain $\sum_{k=0}^na_k=\sum_{k=0}^nf_{n-k}f_k$ 
(where Cauchy's product formula has been utilized). Since $f_{n-k}-f_{n-1-k}=f_{n-2-k}$, it is easy to see that
\begin{align*} a_n=\sum_{k=0}^na_k-\sum_{k=0}^{n-1}a_k
&=\sum_{k=0}^nf_{n-k}f_k-\sum_{k=0}^{n-1}f_{n-1-k}f_k \\
&=f_{n-1}+\sum_{k=0}^{n-1}(f_{n-k}-f_{n-1-k})f_k \\
&=f_{n-1}+\sum_{k=0}^{n-2}f_{n-2-k}f_k.
\end{align*}
To get $(i)$, use $S(x;1,1)=(1-x)L(x;1,1)^2=x\left[\frac{x(1-x)}{(1-x-x^2)^2}\right]$. The proof follows.
\end{proof}

In the next result we obtain a recurrence for the divide-difference $S_n(s,t)$.
\begin{Corollary} 
Preserve the notations from Lemma~\ref{L:generating function for S}. 
Write $S_n$ for $S_n(s,t)$. For $n\geq4$, we have
$$
S_n=(s+t)S_{n-1}+(s+t-st)S_{n-2}-(s^2+t^2)S_{n-3}-stS_{n-4}.
$$
\end{Corollary}
\begin{proof} 
Once more, Lemma~\ref{L:generating function for S} implies $(1-sx-tx^2)(1-tx-st^2)S(x;s,t)=x^2-x^3$. 
Equivalently,
$$
[1-(s+t)x-(s+t-st)x^2+(s^2+t^2)x^3+stx^4]\,S(x;s;t)=x^2-x^3.
$$
Now, simply compare the coefficients of $x^n$ on both sides of the last equation.
\end{proof}

Here is an amusing corollary with beautiful symmetry. 
\begin{Corollary} 
For $s,\, t\in \PP$, we have the numerical series evaluation
$$
\sum_{n\geq0}\frac{S_n(s,t)}{(s+t)^{n+1}}=\frac1{st(s+t-1)}.$$
\end{Corollary}
\begin{proof} Corollary 2.6 of~\cite{ACMS} states that 
$$
\sum_n\frac{L_n(s,t)}{(s+t)^{n+1}}=\frac1{t(s+t-1)}.
$$
Thus, 
\begin{align*}
\sum\frac{S_n(s,t)}{(s+t)^{n+1}} &=\sum\frac{L_n(s,t)-L_n(t,s)}{(s-t)(s+t)^{n+1}}\\
&=\frac1{s-t}\left[\frac1{t(s+t-1)}-\frac1{s(s+t-1)}\right]=\frac1{st(s+t-1)}.
\end{align*}
\end{proof}

\begin{Remark} 
Despite the above plethora of facts, one aspect of the symmetric 
functions $S_n(s,t)$ remains undesirable from a combinatorial view point: 
the coefficients are not all non-negative. 
Fortunately, all is not lost because there is a quick fix as will be seen below.
\end{Remark}

Let $\alpha\in\mathbb{N}$. While maintaining the recursive relation for Lucas polynomials, 
we make a slight alteration to the initial conditions: assume 
$L_0(s,t:\alpha)=L_1(s,t:\alpha)=\alpha$. For $n\geq2$, define
$$
L_n(s,t:\alpha)=sL_{n-1}(s,t:\alpha)+tL_{n-2}(s,t:\alpha).
$$
Let $S_n(s,t:\alpha)$  denote the divided-difference polynomial that is 
associated with $L_n(s,t:\alpha)$.

\begin{Theorem} The following hold true:
\begin{enumerate}
\item[(i)] $S_n(s,t:\alpha)=\alpha S_n(s,t:1)$ for all $\alpha\in\mathbb{N}$;

\item[(ii)] $(s+t-1)$ divides $S_n(s,t:\alpha)$ for all $n\in\mathbb{N}$;

\item[(iii)] $\frac{S_n(s,t:\alpha)}{s+t-1}$ has non-negative integral coefficients, only.
\end{enumerate}
\end{Theorem}

\begin{proof} $(i)$ The defining recurrence and initial conditions imply the homogeneity 
$L_n(s,t:\alpha)=\alpha L_n(s,t:1)$. 
From here, it is evident that $S_n(s,t:\alpha)$ inherits the same property.
Routine standard methods give 
$$
L(s,t:1):=\sum_nL_n(s,t:1)x^n=\frac{1-(s-1)x}{1-sx-tx^2}.
$$ 
One can easily verify that $\sum_nS_n(s,t:1)x^n=\frac{(s+t-1)x^3}{(1-sx-tx^2)(1-tx-sx^2)}$. 
In particular,
$$
\sum_n\frac{S_n(s,t:1)}{s+t-1}x^n=\frac{x^3}{(1-sx-tx^2)(1-tx-sx^2)}=xL(s,t)L(t,s)
$$
simultaneously demonstrates the divisibility in $(ii)$ as well as the claim in $(iii)$.
\end{proof}

\begin{Remark} 
It is worthwhile to note that $L_n(s,t:1)=L_n(s,t)+tL_{n-1}(s,t)$.
As a result, the modified Lucas polynomials also retain a combinatorial interpretation 
much as the ordinary ones. 
Such as simultaneous tiling of a pair of rectangles, one $1\times (n-1)$ and the other 
$1\times n$, where the latter always begins with a domino.
\end{Remark}

\smallskip
\noindent
\textbf{Acknowledgements.} The second and the third authors are partially supported by 
a Louisiana Board of Regents Research and Development Grant 549941C1.

\end{document}